\documentclass{cmslatex}
\usepackage[paperwidth=7in, paperheight=10in, margin=.875in]{geometry}
 \usepackage[backref,colorlinks,linkcolor=red,anchorcolor=green,citecolor=blue]{hyperref}
\usepackage{amsfonts,amssymb}
\usepackage{amsmath}
\usepackage{graphicx}
\usepackage{cite}
\usepackage{enumerate}

\renewcommand{\theequation}{\arabic{section}.\arabic{equation}}

\usepackage{setspace}  
\usepackage{caption}
\usepackage{listings}
\usepackage{tabularx}
\usepackage{xcolor}
\usepackage{indentfirst}
\usepackage{subcaption}

\usepackage{multirow}
\usepackage{overpic}
\usepackage{diagbox}

\usepackage{verbatim}

\usepackage[normalem]{ulem}

\usepackage{tikz}

\usepackage[title]{appendix}

\usepackage{bm}

\usepackage[font=footnotesize,labelfont=bf]{caption}

\makeatletter
\@addtoreset{equation}{section}
\makeatother

\renewcommand{\thefigure}{\arabic{section}.\arabic{figure}}
\renewcommand{\thetable}{\arabic{section}.\arabic{table}}
\renewcommand{\theequation}{\arabic{section}.\arabic{equation}}

\newcommand\dt{\Delta t}

\newcommand\dd{\mathrm{d}}
\newcommand\pp{\partial}

\usepackage{mhchem}

\begin{document}

\title{On a positive-preserving, energy-stable numerical scheme to mass-action kinetics with detailed balance}


\author{Chun Liu\thanks{Department of Applied Mathematics, Illinois Institute of Technology, Chicago, IL 60616, USA. (cliu124@iit.edu)}  \and Cheng Wang\thanks{Mathematics Department, University of Massachusetts--Dartmouth, North Dartmouth, MA 02747, USA,
(cwang1@umassd.edu).}\and Yiwei Wang\thanks{Corresponding author. Department of Mathematics, University of California--Riverside, Riverside, CA 92521, USA. (yiweiw@ucr.edu). }}




\maketitle

\begin{abstract}
In this paper, we provide a detailed theoretical analysis of the numerical scheme introduced in \emph{J. Comput. Phys. 436 (2021) 110253} for the reaction kinetics of a class of chemical reaction network that satisfies detailed balance condition. In contrast to conventional numerical approximations, which are typically constructed based on ordinary differential equations (ODEs) for the concentrations of all involved species, the scheme is developed using the equations of reaction trajectories, which can be viewed as a generalized gradient flow of a physically relevant free energy. The unique solvability, positivity-preserving, and energy-stable properties are proved for the general case involving multiple reactions, under a mild condition on the stoichiometric matrix.
\end{abstract}

\section{Introduction}


Chemical reactions played an important role in many physical, chemical, and biological processes \cite{anderson2015lyapunov, duke1999molecular, keener1998mathematical, pearson1993complex}. Mathematically, the reaction kinetics are often described by a system of nonlinear ODEs in terms of concentrations of all involved species \cite{keener1998mathematical}.

Consider a chemical reaction network (CRN) consisting of $N$ species ${ X_1, X_2, \ldots X_N }$ and $M$ reversible chemical reactions:
\begin{equation}\label{RN1}
\alpha_{1}^{l} X_1 + \alpha_{2}^{l}X_2 + \ldots \alpha_{N}^{l} X_N \ce{<=>} \beta_{1}^{l} X_1 + \beta_{2}^{l}X_2 + \ldots \beta_{N}^{l} X_N, \quad l = 1, \ldots, M \ ,
\end{equation}
where  $\alpha^l_i, \beta^l_i \geq 0$ are stoichiometric coefficients for $l$-th reaction. The reaction kinetics is often formulated as \cite{formaggia2011positivity}
\begin{equation}\label{Reaction_ODE}
\frac{\dd {\bm c}(t)}{\dd t} = {\sf S} {\bm r}({\bm c}).
\end{equation}
Here, ${\bm c}(t) = (c_1, c_2, \ldots, c_N)^{\rm T} \in \mathbb{R}^N_{+}$ represents the concentrations of all involved species,
${\bm r}({\bm c}) = (r_1 ({\bm c}), r_2 ({\bm c}), \ldots r_M ({\bm c})) \in \mathbb{R}^M$ denotes the reaction rates of the $M$ reactions, and ${\sf S} \in \mathbb{R}^{N \times M}$ is the stoichiometric matrix, where each element $S_{il}$ is defined as $\beta^l_i - \alpha^l_i$. It is often assumed that $N > M$ and $\text{rank}({\sf S}) = M$ \cite{formaggia2011positivity}. The latter assumption indicates that the $M$ reactions are linearly independent in this reaction network. Under this assumption, we have
\begin{equation}
 \frac{\dd}{\dd t} ({\bm \gamma} \cdot {\bm c}) = {\bm \gamma} \cdot ({\sf S} {\bm r}) = ({\sf S}^{\rm T} {\bm \gamma}) \cdot {\bm c} = 0, \quad {\bm \gamma} \in \mathrm{Ker} \mathsf{S}^{\rm T} , 
\end{equation}
which indicates that the reaction kinetics (\ref{Reaction_ODE}) employs $N - \mathrm{rank}( {\sf S})$ conserved quantities.

The reaction rate for the $l$-th reaction,  $r_l({\bm c})$, is often expressed as the difference between the forward and backward reaction rates, denoted as $r_l^+ ({\bm c} )$ and $r_l^- ({\bm c})$, i.e., $r_l({\bm c}) = r_l^+ ({\bm c} ) - r_l^- ({\bm c} )$. These rates $r_l^{\pm} ({\bm c} )$ are commonly specified by so-called the law of mass action (LMA) \cite{keener1998mathematical}. The empirical law states that the reaction rate is directly proportional to the product of the concentrations of the reactants, i.e., 
\begin{equation}\label{LMA_1}
r_l^+ (\bm{c}) = k_{l}^+ {\bm c}^{{\bm \alpha}^l} \ , \quad r_l^- ({\bm c}) =  k_{l}^- {\bm c}^{{\bm \beta}^l} \ ,
\end{equation}
where ${\bm c}^{{\bm \alpha}^l}  = \prod_{i=1}^N c_i^{\alpha_i^l}, ~{\bm c}^{{\bm \beta}^l}  = \prod_{i=1}^N c_i^{\beta_i^l}.$  Consequently, the reaction kinetics equation \eqref{Reaction_ODE} is generally a highly nonlinear ODE system.

At a numerical level, solving the reaction kinetics equation (\ref{Reaction_ODE}) is often challenge, primarily due to the stiffness and nonlinearity \cite{formaggia2011positivity}. Moreover, many standard ODE solvers may fail to preserve the basic physical properties of the original system, such as the positivity of ${\bm c}$ and the intrinsic conservation laws. Although there has been a long history of developing robust numerical methods for reaction kinetics \cite{bruggeman2007second, burchard2003high,formaggia2011positivity, sandu2001positive} to preserve the positivity, as well as the conservation property, a significantly small step-size is often needed for most existing methods.  

It has been well-known that if the reaction kinetics (\ref{Reaction_ODE}) with LMA (\ref{LMA_1})
satisfies the {\bf detailed balance} condition, i.e., there exists a positive equilibrium point ${\bm c}^{\infty} \in \mathbb{R}^N_{+}$, such that
\begin{equation}\label{DB_C}
k_{l}^+ ({\bm c}^{\infty} )^{{\bm \alpha}^l} =  k_{l}^- ({\bm c}^{\infty})^{{\bm \beta}^l}, 
\end{equation}
the reaction kinetics (\ref{Reaction_ODE}) admits a Lyapunov function or free energy \cite{anderson2015lyapunov, mielke2011gradient, wei1962axiomatic}, given by
\begin{equation}\label{free_energy_c}
\mathcal{F}[c] = \sum_{i=1}^N c_i (\ln (c_i / c_i^{\infty}) - 1)\ .
\end{equation}

Under the detailed balance condition (\ref{DB_C}), it was shown in \cite{wang2020field} that the system can be viewed as a generalized gradient flow of the reaction trajectory ${\bm R} \in \mathbb{R}^M$ \cite{oster1974chemical, wang2020field}, which accounts for the ``number'' of forward chemical reactions that have occurred by time $t$, with respect to the free energy (\ref{free_energy_c}). More precisely, for the general reaction network (\ref{RN1}), one can introduce a reaction trajectory ${\bm R}(t) \in \mathbb{R}^M$, and ${\bm c}(t)$ will be determined by the kinematics 
\begin{equation}
{\bm c}(t) = {\bm c}(R(t)) = {\bm c}_0 + {\sf S} {\bm R}(t),
\end{equation}
where  
${\sf S} = (S_{il}) \in \mathbb{R}^{N \times M} $ is the stoichiometric matrix and ${\bm c}_0$ is the initial concentration. 
Subsequently, the reaction kinetics with LMA (\ref{LMA_1}) can be viewed as a generalized gradient flow of ${\bm R}$, satisfying the energy-dissipation law
\begin{equation}\label{ED_R}
\frac{\dd}{\dd t} \mathcal{F}[{\bm c}({\bm R})] = - D({\bm R}, \dot{\bm R}), \quad D({\bm R}, \dot{\bm R}) =  \sum_{l=1}^M \dot{R}_l \ln \left( \frac{\dot{R_l}}{k_l^{-}  ({\bm c}({\bm R}))^{{\bm \beta}^l}} + 1 \right)\ , 
\end{equation}
where $D({\bm R}, \dot{\bm R})$ is the rate of energy dissipation. 
Indeed, by a standard variational procedure, one can show that $R_l(t)$ satisfies an nonlinear ODE
\begin{equation}\label{Eq_R}
  \ln \left( \frac{\dot{R_l}}{k_l^{-}  ({\bm c}({\bm R}))^{{\bm \beta}^l}} + 1 \right) = - \frac{\delta \mathcal{F}}{\delta R_l}, \quad \frac{\delta \mathcal{F}}{\delta R_l} = \sum_{i=1}^N S_{li} \mu_i, \quad l = 1, 2, \ldots M.
\end{equation}
where $\mu_i = \frac{\delta \mathcal{F}}{\delta c_i}$ is the chemical potential of $i$-the species, $\frac{\delta \mathcal{F}}{\delta R_l}$ is known as the chemical affinity of $l$-th chemical reaction \cite{kondepudi2014modern}. Using (\ref{Eq_R}), one can rewrite (\ref{Eq_R}) as
\begin{equation}
\dot{R}_l = k_{l}^+ ({\bm c} ({\bm R}) )^{{\bm \alpha}^l} -  k_l^{-}  ({\bm c}({\bm R}))^{{\bm \beta}^l} \ ,
\end{equation} 
which is the LMA. Reaction kinetics beyond the law of mass action can be obtained by choosing the dissipation in (\ref{ED_R}) differently. We refer the interested readers to \cite{wang2020field} for more detailed discussions. It is worth mentioning that, unlike mechanical systems, $D({\bm R}, \dot{\bm R})$ is no longer quadratic  in terms of $\dot{\bm R}$ \cite{wang2020field}. However, near chemical equilibrium, i.e., $\dot{R}_l \approx 0,\ \forall l$, we have $D({\bm R}, \dot{\bm R}) \approx \sum_{l=1}^M |\dot{R}_l|^2 / (k_l^{-}  {\bm c}^{{\bm \beta}^l}).$ Hence, the  linear response assumption is still valid at the last stage of chemical reactions \cite{de2013non}.

The variational formulation (\ref{ED_R}) indicates that the reaction kinetics with the detailed balance condition can be viewed as a generalized  gradient flow of the reaction trajectory.  As a consequence, most numerical techniques for an $L^2-$gradient flow 
can be effectively applied to the reaction kinetics systems of this type. In \cite{liu2021structure}, the authors proposed a numerical scheme that discretizes the reaction trajectory equation (\ref{Eq_R}) directly (see Section 2 for details). The unique solvability, unconditional energy stability and the positivity-preserving property are established for the case with $M = 1$. The convergence analysis has been provided in~\cite{liu2022convergence}, and an extension to the second order numerical algorithm has been reported in~\cite{liu2022second}. 


Although numerical tests in \cite{liu2021structure,liu2022second} indicate the proposed numerical schemes work for cases with $M > 1$,  the theoretical analysis in \cite{liu2021structure,liu2022second} is limited to  the case of $M=1$. The aim of this short note is to provide a theoretical justification for the proposed numerical scheme, in particular in terms of the positivity-preserving property, unique solvability, and unconditional energy stability for the multiple reaction case, with $M > 1$. To clarify the idea, we only write down the details for the case with $M =2$ and $N = 4$, but the proof strategy works for the general case where $N \geq M$ and ${\rm rank} ({\sf S}) = M$. 

The remainder of this paper is organized as follows. The structure-preserving numerical scheme is recalled in Section~\ref{sec:numerical scheme}. The theoretical justification of positivity-preserving analysis and unique solvability is provided in Section~\ref{sec:positivity}. 








\section{The structure-preserving numerical discretization} \label{sec:numerical scheme} 

In this section, we briefly review the numerical scheme for the reaction kinetics, proposed in \cite{liu2021structure}. Instead of solving the reaction kinetics equation for the concentrations of all involved species (\ref{Reaction_ODE}), the numerical discretization is constructed on the reaction trajectory equation (\ref{Eq_R}), which can be viewed as a generalized gradient flow of ${\bm R}$. Similar to an $L^2$-gradient flow, a first-order semi-implicit discretization to
 (\ref{Eq_R}) can be written as
\begin{equation}\label{scheme_Mul_R}
 \ln \left(  \frac{R_l^{n+1} - R_l^n}{ k_l^{-}  ({\bm c}^n)^{{\bm \beta}^l} \dt} + 1 \right) = -  \frac{\delta \mathcal{F}}{\delta R_l} ( {\bm R}^{n+1})\ ,   \quad  1 \leq l \leq M , 
\end{equation}
where ${\bm c}^n =  {\bm c}_0 + {\sf S} {\bm R}^n$ and $\dt$ is the temporal step-size. Although this equation is nonlinear with respect to $R_l^{n+1}$, its variational structure allows us to reformulate it as an optimization problem:
\begin{equation}\label{Minimization_R}
  \begin{aligned}
&  {\bm R}^{n+1} = \text{argmin}_{{\bm R} \in \mathcal{V}^n} J^n (\bm{R}), \quad J^n({\bm R}) =  d_R^2 ({\bm R}, {\bm R}^n)  + \mathcal{F} [ {\bm c} ({\bm R})] . 
  \end{aligned}
\end{equation}
Here, ${\bm c}({\bm R}) = {\bm c}_0 + {\sf S} {\bm R}$, $d_R^2({\bm R}, {\bm R^n})$ is a function measuring the difference between ${\bm R}$ and ${\bm R}^n$, defined as
\begin{equation}
d_R^2 ({\bm R}, {\bm R}^n) = \sum_{l=1}^M \left( (R_l - R^n_l + k_l^{-}  ({\bm c}^n)^{{\bm \beta}^l} \dt ) \ln \left(  \frac{R_l - R^n_l}{ k_l^{-}  ({\bm c}^n)^{{\bm \beta}^l}  \dt} + 1 \right) - (R_l -  R^n_l)) \right) , 
\end{equation}
and the admissible set is given by 
\begin{equation}
\mathcal{V}^n = \{  {\bm R} \in \mathbb{R}^M ~|~ {\bm c}_0 + {\sf S} {\bm R} \in \mathbb{R}^N_{+},~~ R_l - R^n_l + k_l^{-}  ({\bm c}^n)^{{\bm \beta}^l} \dt > 0,~~ J^n[{\bm R}^{n+1}] \leq J^n ({\bm R}^n \}.
\end{equation}
Of course, $\mathcal{V}^n$ is a non-empty set, since ${\bm R}^n \in \mathcal{V}^n$. Moreover, noticing that $d_R^2 ({\bm R}, {\bm R}^n) \rightarrow \infty$ if $\| {\bm R} \| \rightarrow \infty$ and $\mathcal{F} [c({\bm R})]$ is bounded from below, we conclude that $\mathcal{V}^n$ is a bounded subset of $\mathbb{R}^M$.
The set $\{ {\bm R} \in \mathbb{R}^M ~|~ {\bm c}_0 + {\sf S} {\bm R} \in \mathbb{R}^N_{+} \}$ is called  stoichiometric compatibility class for the initial condition $c_0$ \cite{anderson2015lyapunov}. It is straightforward to verify that 
\begin{equation}
\frac{\delta J^n ({\bm R})}{\delta R_l} =   \ln \left(  \frac{R_l - R_l^n}{ k_l^{-}  ({\bm c}^n)^{{\bm \beta}^l} \dt} + 1 \right) + \frac{\delta \mathcal{F}}{\delta R_l}, \quad \forall l . 
\end{equation}
Hence, a critical point of $J^n({\bm R})$ in $\mathcal{V}^n$ gives a solution of the nonlinear equation (\ref{scheme_Mul_R}). 


\begin{remark}
  It is worth mentioning that an explicit treatment of ${\bm R}$ in the term $k_l^{-}  ({\bm c}({\bm R}))^{{\bm \beta}^l}$ turns out to be crucial, and it enables the definition of $d_R^2 ({\bm R}, {\bm R}^n)$. 
 Moreover, if $\frac{R_l - R^n_l}{ k_l^{-}  ({\bm c}^n)^{{\bm \beta}^l}  \dt}$ is small for any $l$, we observe the following Taylor expansion: 
\begin{equation}
  d_R^2 ({\bm R}, {\bm R}^n) \approx \sum_{l=1}^m \frac{1}{  k_l^{-}  ({\bm c}^n)^{{\bm \beta}^l}  \dt } (R_l - R_l^n)^2 + \mbox{higher order terms} . 
\end{equation}
Therefore, the numerical scheme is a natural generalization for the minimizing movement scheme for an $L^2$-gradient flow.
\end{remark}

It is straightforward to prove the following unconditional energy stability result by using the property of $d_R^2 ({\bm R}, {\bm R}^n)$. 
\begin{proposition}
  If ${\bm R}^{n+1}$ is a global minimizer of $J^n ({\bm R})$ in $\mathcal{V}^n$, then the numerical scheme is unconditionally energy stable.
\end{proposition}

\begin{proof} 
Define $f(x) = (x + a) \ln (x / a + 1) - x$, where $a > 0$ is a given constant. It is clear that $f(x)$ is a monotonic increasing function of $x$ for $x \geq 0$ and $f(0) = 0$. Consequently, $d_R^2 ({\bm R}, {\bm R}^n) \geq 0$ in $\mathcal{V}^n$ and $d_R^2 ({\bm R}, {\bm R}^n) = 0$ if and only if ${\bm R} = {\bm R}^n$.
Hence, if  ${\bm R}^{n+1}$ is a global minimizer of $J^n ({\bm R})$ in $\mathcal{V}^n$, we have 
\begin{equation}\label{Energy_stable}
  \mathcal{F}({\bm R}^{n+1}) \leq J^n ({\bm R}^{n+1}) \leq J^n ({\bm R}^{n}) = \mathcal{F}({\bm R}^n),
\end{equation}
which gives the unconditional energy stability. 
\end{proof}


\section{The positivity-preserving analysis and unique solvability}  \label{sec:positivity} 

The main theoretical question associated with the numerical scheme (\ref{Minimization_R}) is the existence and uniqueness of the global minimizer of $J^n ({\bm R})$ in $\mathcal{V}^n$. This property has been proved in \cite{liu2021structure} for the case with $M = 1$.  In this section, we demonstrate that the result can be generalized to the general case of $M \leq N$ and ${\rm rank} ({\sf S}) = M$. More precisely, we have the following theorem.

\begin{theorem}  \label{Multiple field-positivity} 
If $M \le N$ and ${\rm rank}(\sf S) = M$, then given ${\bm R}^n \in \mathbb{R}^M$, with ${\bm c}^n = {\bm c}_0 + {\sf S} {\bm R}^n   \in \mathbb{R}^N_{+}$, there exists a unique solution ${\bm R}^{n+1} \in \mathcal{V}^n$  
for the numerical scheme~\eqref{scheme_Mul_R}. 
\end{theorem} 



To prove this result, we first observe the following lemma. 

\begin{lemma}
  If $M \leq N$ and $\rm{rank}({\bm \sigma }) = M$, $0 < c_i \leq A^*$, then $J^n({\bm R})$ is a convex function of ${\bm R}$ in $\mathcal{V}^n$.
\end{lemma}
\begin{proof}
Denote $g({\bm R}) =  d_R^2({\bm R}, {\bm R}^n)$. A direct calculation implies that 
\begin{equation}
\frac{\pp^2 g}{\pp R_l^2} = \frac{1}{ R_l - R^n_l + k_l^{-}  ({\bm c}^n)^{{\bm \beta}^l} \dt } > 0, \quad  \frac{\pp^2 g}{\pp R_l \pp R_k} = 0~\text{if}~l \neq k ,  \quad \forall {\bm R} \in \mathcal{V}^n . 
\end{equation}
Hence $g({\bm R})$ is a convex function of ${\bm R}$ over $\mathcal{V}^n$. For $\mathcal{F}[c({\bm R})]$, we recall that ${\bm c} ({\bm R}) = {\bm c}_0 + {\sf S} {\bm R}$, and a direct calculation gives 
  \begin{equation}
  \nabla_R^2 F({\bm R}) = {\sf S}^{\rm T} (\nabla^2_{\bm c} F({\bm c})) {\sf S}  , 
  \end{equation}
  where ${\nabla^2_{\bm c} F} = \rm{diag} (\frac{1}{c_1}, \frac{1}{c_2}, \ldots \frac{1}{c_N}).$ 
  Because of the defintion of $\mathcal{V}^n$, we have a uniform bound of $c_i$, i.e., $0 < c_i \leq A^*$, which results in
  $$\lambda_{\rm min} ( \nabla_R^2 F) \geq \frac{1}{A^*} \lambda_{\rm min} ( {\sf S}^{\rm T}  {\sf S} ) > 0 . $$
Henceforth, $F({\bm R})$ is a convex function of ${\bf R}$ over $\mathcal{V}^n$.
 \end{proof}


Since $\mathcal{V}^n$ is a bounded set of $\mathbb{R}^M$, and $J^n ({\bm R})$ is a convex function of ${\bm R}$ in $\mathcal{V}^n$, then there exists a unique minimizer of $J^n({\bm R})$ in $\mathcal{V}^n$. The key point of the proof is to show that the minimizer of $J^n({\bm R})$ over $\mathcal{V}^n$ cannot occur on the boundary of $\mathcal{V}^n$, so that the global minimizer of $J^n ({\bm R})$ is a critical point of $J^n ({\bm R})$, which turns out to be a solution of (\ref{scheme_Mul_R}). 

To illustrate this idea, we present the case with $M = 2$ and $N = 4$. The analysis can be extended to different values of $M$ and $N$ following the same strategy.  First, we define a linear transformation of $R_i$
\begin{equation}
\tilde{R}_1 = c_{1}^0 + \sum_{j=1}^2 S_{1j}R_j, \quad 
\tilde{R}_2 = c_{2}^0 + \sum_{j=1}^2 S_{2j} R_j . 
\end{equation}
The positive stoichiometric compatibility class can be written in terms of $\tilde{R}_1$ and $\tilde{R}_2$, given by
$$\{ (\tilde{R}_1, \tilde{R}_2) | \tilde{R}_i > 0, c_3^0 + \sum_{j=1}^2   \tilde{S}_{3j}\tilde{R}_j > 0, \quad c_4^0 + \sum_{j=1}^2   \tilde{S}_{4j}\tilde{R}_j > 0  \}, $$
where $\tilde{S}_{3j}$ and $\tilde{S}_{4j}$ are transformed stoichiometric coefficients in terms of  $\tilde{R}_1$ and $\tilde{R}_2$.


\begin{example}
We consider a concrete example of a reaction network
\begin{equation}
X_1 + 2 X_2 \ce{<=>} X_3, \quad X_2 + X_3 \ce{<=>} 2 X_4 . 
\end{equation}
In turn, the stoichiometric matrix is given by
\begin{equation}
{\sf S} = 
\begin{pmatrix}
& -1 & 0 \\
& -2 & 1 \\
& 1 & -1 \\
& 0 & 2 \\
\end{pmatrix} . 
\end{equation}
Assume that ${\bm c}_0 = (1, 1, 1, 1)^{\rm T}$, then the positive stoichiometric compatibility class corresponds to the set in the reaction space
\begin{equation*}
\{ (R_1, R_2) |     1 - R_1 > 0, 1 - 2 R_1 + R_2 > 0, 1 + R_1 - R_2 \geq 0, 1 + 2 R_2 > 0 \} . 
\end{equation*}
In this case, the linear transformation of $R_i$ is defined as 
\begin{equation*}
\tilde{R}_1 = 1 - R_1, \quad \tilde{R}_2 = 1 - 2 R_1 + R_2 , 
\end{equation*}
and the stoichiometric compatibility class becomes
$$ \{ (\tilde{R}_1, \tilde{R}_2) | \, \tilde{R}_i > 0, 3 - 3 \tilde{R}_1 + \tilde{R}_2 > 0, \quad -1 + 4 \tilde{R}_1 -  2 \tilde{R}_2 > 0. \}$$ It is important to notice that the boundary of the stoichiometric compatibility class turns out to be $c_3 = 0$ and (or) $c_4 = 0$.
\end{example}

Without ambiguity, we omit the tilde notation in the following description. With a linear transformation, the kinematics can be rewritten as
\begin{equation}
  c_1 = R_1, \quad c_2 = R_2, 
  \quad c_3 = c_0^3 + S_{31} R_1 + S_{32} R_2, \quad  c_4 = c_0^4 + S_{41} R_1 + S_{42} R_2,
\end{equation}
and the free energy becomes
\begin{equation}
\mathcal{F}[R_1, R_2] = R_1 \ln \left(  \frac{R_1}{c_1^{\infty}} - 1 \right) + R_2 \ln \left(  \frac{R_2}{c_2^{\infty}} - 1 \right) + c_3 \ln \left(  \frac{c_3}{c_3^{\infty}} - 1 \right)  + c_4 \ln \left(  \frac{c_4}{c_4^{\infty}} - 1 \right).
\end{equation}



Denote $\gamma^n_l =   R_l^n -  k_l^{-}({\bm c}^n) \Delta t$.
Since $R_l^n > 0$, it is clear that $\gamma^n_l \geq 0$ for $\Delta t$ significantly small. 
Without loss of generality, we assume that $\gamma^n_l = 0$. In the case where $\gamma^n_l > 0$, we can adopt our approach to work on $R_l - \gamma^n_l$ instead. Moreover, to simplify the presentation, we take  $\bar{c}_0^3 =  \bar{c}_0^4 = 1$. Then the admissible set is given by
\begin{equation}
  \mathcal{V}^n = \mathcal{V}^n_0 \cap \{ R | J^n({\bm R}) \leq J^n({\bm R}^n) \} \ ,  
  \end{equation}
%
where $ \mathcal{V}^n_0 = \{  (R_1, R_2)  ~~|~~   R_1 > 0, R_2 > 0, 1 + S_{31} R_1 + S_{32} R_2 > 0, c_1 + S_{41} R_1 + S_{42} R_2 > 0   \}.$
Figure \ref{admissible_set}(a)-(i) displays the possible geometry of the set $\mathcal{V}^n_0$.
\begin{figure}
  \centering
  \includegraphics[width = 0.9 \linewidth]{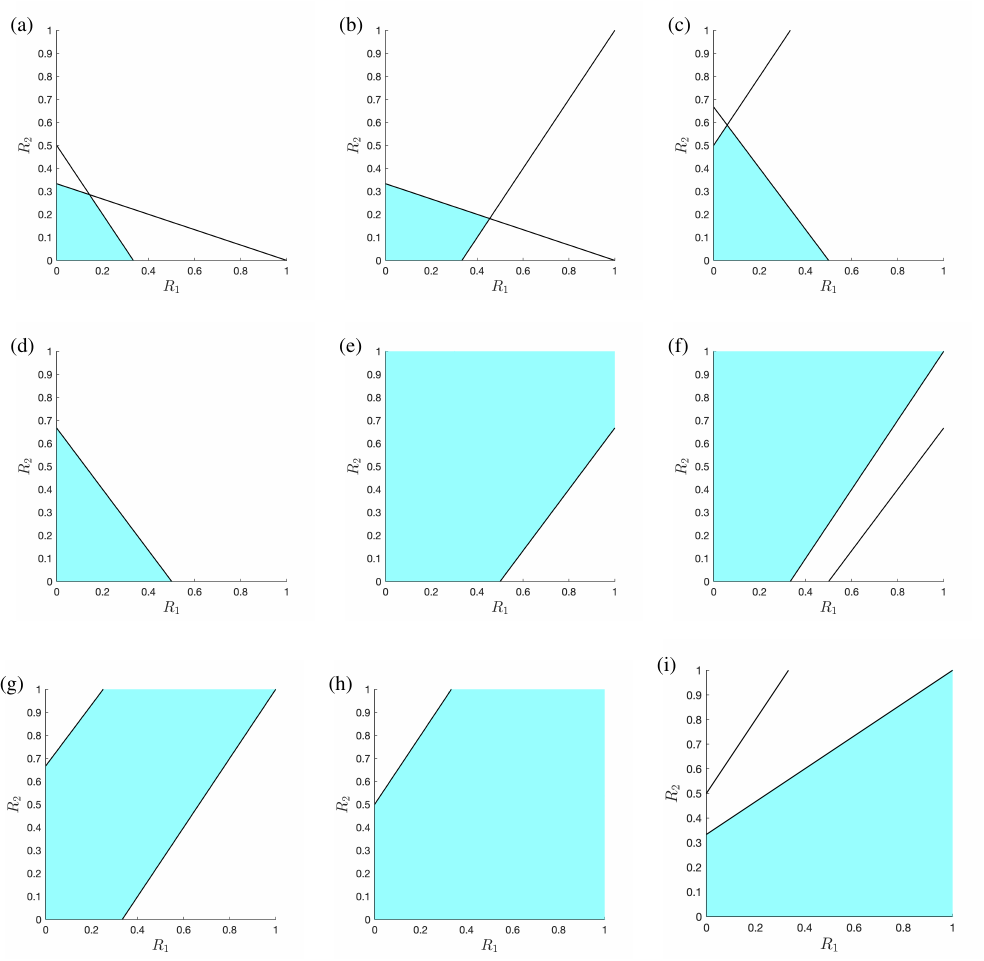}
  \caption{Illustration of the transformed stoichiometric compatibility set $\mathcal{V}^n_0 = \{  (R_1, R_2)  ~|~   R_1 > 0, R_2 > 0, c_3(R_1, R_2) > 0, c_4(R_1, R_2) > 0   \}$ according to the signs of $(S_{31}, S_{32}, S_{41}, S_{42})$, (a) $(-,-,-,-)$; (b) $(-,+,-,-)$ or $(-,-,-,+)$; (c) $(+,-,-,-)$ or $(-,-,+,-)$; (d) $(+,+,-,-)$ or $(-,-,+,+)$; (e) $(+,+,-,+)$ or $(-,+,+,+)$; (f) $(-,+,-,+)$; (g) $(-,+,+,-)$ or $(+,-,-,+)$; (h) $(+,-,+,+)$ or $(+,+,+,-)$;  (i) $(+,-,+,-).$  The case $(+,+,+,+)$ is not shown. }\label{admissible_set}
  \end{figure}

  It is important to note that the set $\mathcal{V}^n_0$ may not necessarily be bounded. Hence, it is crucial to consider $\mathcal{V}^n_0 \cap \{ R | J^n({\bm R}) \leq J^n({\bm R}^n) \}$. The boundedness of ${\bm R}$ come from the condition $J^n({\bm R}) \leq J^n({\bm R}^n)$. Due to this bound, we have $ 0 < c_i ({\bm R}) < A^*,~ \forall {\bm R} \in \mathcal{V}^n$ for some constant $A^*$.
  
  To show that the global minimizer of $J^n(R_1, R_2)$ over $\mathcal{V}^n$ cannot be obtained on the boundary, we only need to consider following possible boundaries
  \begin{equation}
    \begin{aligned}
    & \Gamma_1 = \{ (R_1, R_2) | R_1 = 0  \},  \Gamma_2 = \{ (R_1, R_2) | R_2 = 0  \}, \\ 
    & \Gamma_3 = \{ (R_1, R_2) | c_3(R_1, R_2) = 0  \},  \Gamma_4 = \{ (R_1, R_2) | c_4(R_1, R_2) = 0  \} . 
    \end{aligned}
  \end{equation}
To this end, the following subset of $\mathcal{V}^n$ is taken into consideration: 
\begin{equation}
\mathcal{V}^n_{\delta} = \{ (R_1, R_2) \in \mathcal{V}^n  ~|~ R_1, R_2 \geq g(\delta), c_3, c_4 \geq \delta  \} \subset \mathcal{V}^n . 
\end{equation}
Let
\begin{equation}
  \begin{aligned}
  & \Gamma_1^{\delta} = \{ (R_1, R_2) | R_1 = g(\delta) \},  \Gamma_2^{\delta} = \{ (R_1, R_2) | R_2 = g(\delta)  \}, \\ 
  & \Gamma_3^{\delta} = \{ (R_1, R_2) | c_3(R_1, R_2) = \delta  \},  \Gamma_4^{\delta} = \{ (R_1, R_2) | c_4(R_1, R_2) = \delta  \} , 
  \end{aligned}
\end{equation}
where $g(\delta)$ is a certain function that will be specified later. We only need to prove that the minimizer of $J^n$ over $\mathcal{V}^n_{\delta}$ could not occur on $\Gamma_i^{\delta} \cap \mathcal{V}^n$ ($i=1,\ldots 4$), if $\delta$ is taken significantly small. The strategy is to first assume that the minimizer of $J^n(R_1, R_2)$ over $\mathcal{V}^n_{\delta}$ occurs at a boundary point $(R_1^*, R_2^*) \in \Gamma_i^{\delta}$ for some $i$. In turn, if one can find $(R_1', R_2') \in (\mathcal{V}^n_{\delta})^{\circ}$ that $J^n(R_1', R_2') < J^n (R_1^*, R_2^*)$, then it leads to a contradiction. Such a strategy follows similar ideas as the positivity-preserving analysis reported in~\cite{chen2019positivity,liu2021structure}. At the beginning, we calculate the partial derivatives of $J^n(R_1, R_2)$ with respect to $R_1$ and $R_2$  . The derivatives are given by  
\begin{equation}
  \begin{aligned}
  & \frac{\pp J^n}{\pp R_1} =  \ln \left( \frac{R_1 - R_1^n}{ k_1^- ({\bm c}^n)^{\beta_1} \Delta t} + 1 \right) + \ln \left( \frac{R_1}{c_1^{\infty}} \right) + S_{31} \ln \left( \frac{c_3({\bm R})}{c_3^{\infty}} \right) + S_{41} \ln \left( \frac{c_4({\bm R})}{c_4^{\infty}} \right) , \\
  & \frac{\pp J^n}{\pp R_2} =  \ln \left( \frac{R_2 - R_2^n}{ k_2^- ({\bm c}^n)^{\beta_2} \Delta t}  + 1\right) + \ln \left( \frac{R_2}{c_2^{\infty}} \right) + S_{32} \ln \left( \frac{c_3({\bm R})}{c_3^{\infty}} \right) + S_{42} \ln \left( \frac{c_4({\bm R})}{c_4^{\infty}} \right) . 
  \end{aligned}
  \end{equation}
We will use these derivatives extensively in the subsequent analysis..



It is noticed that $\Gamma_1^{\delta} \cap \mathcal{V}^n$ and $\Gamma_2^{\delta} \cap \mathcal{V}^n$ are always two boundary sections of $\mathcal{V}^n_{\delta}$. We first consider the boundaries $\Gamma_1^{\delta} \cap \mathcal{V}^n$ and $\Gamma_2^{\delta} \cap \mathcal{V}^n$, by assuming the minimizer occurs at $R_1^* = g(\delta)$ or $R_2^* = g(\delta)$.
Because of the symmetry, we only need to consider the case that $R_1^* = g(\delta)$, which in turns indicates that 
\begin{equation}
  \begin{aligned}
  \frac{\pp J^n}{\pp R_1} |_{(R_1^*, R_2^*)} & =\ln \left( g(\delta) \right) +  \ln (g(\delta)) + S_{31} \ln (c_3^*) + S_{41} \ln (c_4^*) + Q_1 , 
  \end{aligned}
\end{equation}
where $Q_1 =  - \ln (k_1^- ({\bm c}^n)^{\beta_1} \Delta t)  - \ln c_1^{\infty} - S_{31} \ln c_3^{\infty} - S_{41} \ln c_4^{\infty}$ is a constant. Recall that $\delta \leq c_3^* \leq A^*$ and $\delta \leq c_4^* \leq A^*$, and we always have
 $$S_{31} \ln (c_3^*) + S_{41} \ln (c_4^*) \leq - |S_{31}| \ln \delta - |S_{41}| \ln \delta , $$
for some significantly small $\delta$. 
One can always choose $g(\delta) = \delta^{\alpha}$ for some positive $\alpha$ such that $\frac{\pp J^n}{\pp R_1} |_{(R_1^*, R_2^*)} < 0$ with $\delta$ being significantly small. Then we can find $R_1' > R_1^* = g(\delta)$ such that $J^n(R_1', R_2^*) < J^n (R_1^*, R_2^*)$. Because of the fact that $(R_1', R_2^*) \in \mathcal{V}_{\delta}$, this contradicts with the assumption that $(R_1^*, R_2^*)$ is a minimizer.

Next, we look at the possible boundary sections $\Gamma_3^{\delta} \cap \mathcal{V}^n$ and $\Gamma_4^{\delta} \cap \mathcal{V}^n$,. The following different cases have to be discussed separately.

\noindent {\bf Case 1.} \,  $S_{31} < 0$,  $S_{32} < 0$,  $S_{41} < 0$, $S_{42} <0$:
In this case, the admissible set is sketched in Figure \ref{admissible_set}(a), and $V_{\delta^n}$ is the closed bounded set. 

We first assume that the minimizer occurs on $\Gamma_3^{\delta} \cap \mathcal{V}^n$.
Since $S_{31} < 0$,  $S_{32} < 0$, we see that either $| S_{31} |  R_1^* \ge \frac{1}{3}$ or $| S_{32} |  R_2^* \ge \frac{1}{3}$, if $\delta$ is significantly small.  Without loss of generality, it is assumed that $| S_{31} |  R_1^* \ge \frac{1}{3}$, so that $R_1^* \ge \frac{1}{- 3 S_{31}} := B_1^*$. Also notice that
\begin{equation}
  \begin{aligned}
  \frac{\pp J^n}{\pp R_1} |_{(R_1^*, R_2^*)}  & =  \ln \left( \frac{R_1 - R_1^n}{ k_1^- ({\bm c}^n)^{\beta_1} \Delta t}  + 1 \right) + \ln \left( \frac{R_1}{c_1^{\infty}} \right) + S_{31} \ln \left(  \frac{\delta}{c_3^{\infty}} \right) + S_{41} \ln \left( \frac{c_4^*}{c_4^{\infty}} \right) \\
  & = \ln \left( R_1 \right) + \ln R_1  + S_{31} \ln (\delta) + S_{41} \ln (c_4^*) + Q_1 \\
  & \geq \ln B_1^* + \ln B_1^* + S_{31} \ln (\delta) + S_{41} \ln A^* + Q_1 , 
  \end{aligned}
\end{equation}
where $Q_1 = - \ln (k_1^- ({\bm c}^n)^{\beta_1} \Delta t)  - \ln c_1^{\infty} - S_{31} \ln c_3^{\infty} - S_{41} \ln c_4^{\infty}$ is a constant, and $A^*$ is the upper bound of $c_4$ in $\mathcal{V}^n \cap \{ {\bm R} ~|~ J^n({\bm R}) \leq J^n({\bm R}^n) \}$. Since $B_1^*$, $A^*$ and $Q_1$ are constants that are independent with $\Delta t$ and $\delta$, we are able to choose $\delta$ significantly small such that 
$\frac{\pp J^n}{\pp R_1} |_{(R_1^*, R_2^*)} > 0$. 
In other words, one can find $\delta < R_1' < R_1$ such that $J(R_1', R_2^*) \leq J(R_1^*, R_2^*)$. The fact that $c_3(R_1', R_2^*) > c_3^* = \delta \in \mathcal{V}^n_\delta$ leads to a contradiction that $(R_1^*, R_2*)$ is a minimizer in $\mathcal{V}_{\delta}$. Using a similar argument, we are able to prove that the minimizer cannot occur at $c_4^* = \delta$, either.

\noindent {\bf Case 2.} \,  $S_{31} < 0$,  $S_{32} < 0$,  $S_{41} < 0$, $S_{42} >0$, which corresponds to Figure \ref{admissible_set}(b).

We first consider the boundary $\Gamma_3^{\delta} \cap \mathcal{V}^n$. On this boundary section, we see that either $R_1^* \geq \frac{1}{- 3 S_{31}} = B_1^*$ or $R_2^* \geq \frac{1}{- 3 S_{32}} = B_2^*$. In addition, denote $B_3^* = \min (B_1^*, -1 / S_{41})$. If $R_1^* \geq B_3^*$, using similar arguments in the previous case, we have 
\begin{equation}
  \begin{aligned}
  \frac{\pp J^n}{\pp R_1} |_{(R_1^*, R_2^*)}  
  & = \ln \left( R_1 \right) + \ln R_1  + S_{31} \ln (\delta) + S_{41} \ln (c_4^*) + Q_1 \\
  & \geq \ln B_3^* + \ln B_3^* + S_{31} \ln (\delta) + S_{41} \ln A^* + Q_1 . 
  \end{aligned}
\end{equation}
In turn, $\delta$ can be chosen significantly small, so that $\frac{\pp J^n}{\pp R_1} |_{(R_1^*, R_2^*)} > 0$. This leads to a contradiction. If $R_1^* \leq B_3^* \leq B_1^*$, we get  $R_2^* \geq B_2^*$, and notice that
\begin{equation}
c_4^* = 1 + S_{41} R_1^* + S_{42} R_2^* \geq 1 + S_{41} B_3^*  + S_{42} B_2^* \geq S_{42} B_2^*, 
\end{equation}
\begin{equation}
  \begin{aligned}
  \frac{\pp J^n}{\pp R_2} |_{(R_1^*, R_2^*)}  & = \ln (R_2^*) + \ln (R_2^*) + S_{41} \ln (\delta) + S_{42} \ln c_4^* + Q_2  \\
  & \geq \ln B_2^* + \ln B_2^* + S_{31} \ln \delta + S_{42} \ln (S_{42} B_2^*) + Q_2.
  \end{aligned}
\end{equation}
Again, since other terms are constants, we can choose $\delta$ significantly small, such that $\frac{\pp J^n}{\pp R_2} |_{(R_1^*, R_2^*)} > 0$. Therefore, one can find $R_2' < R_2^*$, such that $J^n (R_1^*, R_2') < J^n (R_1^*, R_2^*)$, which leads to a contradiction as $c_3(R_1^*, R_2') \in \mathcal{V}_{\delta}$.

Next we consider the case of $c_4^* = \delta$ (and $R_2^* > \delta$). Notice that, by choosing $\delta$ significantly small, we have
\begin{equation}
 S_{41} R_1^* = \delta  - 1 - S_{42} R_2^* \leq \delta - 1, \Rightarrow  R_1^* \geq \frac{-1 + \delta}{S_{41}} . 
\end{equation}
By choosing $\delta$ significantly small, we get $R_1^* \geq - \frac{1}{2 S_{41}} = B_4^*$. Therefore, the following inequality is valid: 
\begin{equation}
  \begin{aligned}
  \frac{\pp J^n}{\pp R_1} |_{(R_1^*, R_2^*)}  
  & \geq \ln B_4^* + \ln B_4^* + S_{31} \ln A^* + S_{32} \ln \delta + Q_1 , 
  \end{aligned}
\end{equation}
so that $\delta$ could be chosen significantly small satisfying $ \frac{\pp J^n}{\pp R_1} |_{(R_1^*, R_2^*)} > 0$.  Combining all these arguments, we conclude that a minimization point cannot occur at either $c_3^* = \delta$ or $c_4^* = \delta$, provided that $\delta$ is sufficiently small, in the case of $S_{31} < 0$,  $S_{32} < 0$,  $S_{41} < 0$, $S_{42} >0$. 
  
Due to the symmetry, the following cases (shown in Fig. \ref{admissible_set}(c)) could be analyzed in a similar manner:
\begin{itemize}
  \item $S_{31} < 0, S_{32} > 0, S_{41} < 0, S_{42} < 0$
  \item $S_{31} > 0, S_{32} < 0, S_{41} < 0, S_{42} < 0$
  \item $S_{31} < 0, S_{32} < 0, S_{41} > 0, S_{42} < 0$
\end{itemize}


\noindent {\bf Case 3.} \,  $S_{31} < 0$,  $S_{32} > 0$,  $S_{41} < 0$, $S_{42} >0$, which corresponds to Figure \ref{admissible_set}(f).  

If a minimization point occurs at $(R_1^*, R_2^*)$ with $c_4^*=  ( 1 + S_{41}  R_1^* + S_{42} R_2^* )  = \delta$, we see that $R_1^* \ge \frac{-1}{S_{41}} := B_4^*$ (since $S_{42} >0$). In turn, the following estimate could be derived: 
\begin{eqnarray} 
 \frac{\pp J^n}{\pp R_1}
   \ge  
   \ln B_4^*  + \ln B_4^*   
   +  S_{31} \ln A^*  
   +  S_{32} \ln \delta  + Q_1 .  
    \label{Two-positive-13} 
\end{eqnarray} 
Again, the value of $\ln B_4^*  + \ln B_4^*  +  S_{31} \ln A^*  +  S_{32} \ln \delta$ becomes a fixed constant with a fixed $\dt$, and we could always choose $\delta$ significantly small such that $\partial_{R_1} J \mid_{(R_1^*, R_2^*)}  > 0$, which makes a contradiction to the assumption that $J (R_1, R_2)$ reaches a minimization point at $(R_1^*, R_2^*)$ over $V_{\delta}$. 
Using similar arguments, a minimization point cannot occur at $(R_1^*, R_2^*)$ with $c_3^*=  1 + S_{31}  R_1^* + S_{32} R_2^*   = \delta$, either, in the case of $S_{31} < 0$,  $S_{32} > 0$,  $S_{41} < 0$, $S_{42} >0$, if $\delta$ is sufficiently small. Because of the symmetry, the case of $S_{31} > 0$,  $S_{32} < 0$,  $S_{41} > 0$, $S_{42} < 0$, as shown in Figure~\ref{admissible_set}(i),  could be analyzed in a similar style (by switching $R_1$ and $R_2$).

\noindent {\bf Case 5.} \,  $S_{31} < 0$,  $S_{32} > 0$,  $S_{41} > 0$, $S_{42} <0$, which corresponds to Figure~\ref{admissible_set}(g).  

If a minimization point occurs at $(R_1^*, R_2^*)$ with $c_3^*= ( 1 + S_{31}  R_1^* + S_{32} R_2^* )  = \delta$, we see that $R_1^* \ge \frac{-1}{S_{31}} := B_5^*$ (since $S_{31} >0$, $S_{32} < 0$). This in turn indicates that 
\begin{eqnarray} 
  \frac{\pp J^n}{\pp R_1} \mid_{(R_1^*, R_2^*)}   \ge   \ln B_5^*  + \ln B_5^*   +  S_{31} \ln A^*  +  S_{32} \ln \delta  + Q_1 .  
\end{eqnarray} 
Again, the uniform bound $c_4^* \le A^*$ has been applied in the derivation. We could always choose $\delta$ significantly small so that $\frac{\pp J^n}{\pp R_1} \mid_{(R_1^*, R_2^*)} > 0$, which makes a contradiction to the assumption that $J (R_1, R_2)$ reaches a minimization point at $(R_1^*, R_2^*)$ over $V_{\delta}$. 
Using similar arguments, a minimization point cannot occur at $(R_1^*, R_2^*)$ with $c_4^*=  1 + S_{41}  R_1^* + S_{42} R_2^*   = \delta$, either, due to the fact that $R_2^*$ is bounded from below. Due to the symmetry, the case of $S_{31} > 0$,  $S_{32} < 0$,  $S_{41} < 0$, $S_{42} > 0$, could be analyzed in a similar fashion. 



\noindent {\bf Case 6.} \,  $S_{31} > 0$,  $S_{32} > 0$. 

In this case, the boundary section $c_3^* = \alpha_3 ( 1 + S_{31} R_1^* + S_{32} R_2^*)  = \delta$ will never be reached, because of the fact that $R_1^* > 0$, $R_2^* >0$. In turn, the four boundary section constraint will be reduced to the three-boundary-section version, and the analysis in the previous cases could be recalled. 

\noindent {\bf Case 7.} \,  $S_{41} > 0$,  $S_{42} > 0$. 

Similarly, the boundary section $c_4^* =   1 + S_{41} R_1^* + S_{42} R_2^*  = \delta$ will never be reached in this case, since $R_1^* > 0$, $R_2^* >0$. Similarly, the four boundary section constraint will be reduced to the three-boundary-section version, and the analysis in the previous cases could be recalled. 

Therefore, a combination of all these cases have demonstrated that, if the minimizer of $J (R_1, R_2)$ could not occur at a boundary point of $V_{\delta}$ where either $c_3 = \delta$ or $c_4 = \delta$, which completes the proof.








\section*{Acknowledgement} 
This work is partially supported by the National Science Foundation (USA) grants NSF DMS-1759536, NSF DMS-1950868 (C. Liu, Y. Wang), NSF DMS-2012669, DMS-2309548 (C. Wang). 


\bibliographystyle{siam}
\bibliography{KCR}

\end{document}